\documentclass[reqno]{amsart}

\usepackage{amssymb}  
\usepackage{enumitem}
\usepackage{tikz}
\usepackage{hyperref}
\usepackage{bbm}

\newcommand*{\mailto}[1]{\href{mailto:#1}{\nolinkurl{#1}}}

\newcommand{\msc}[1]{\href{http://www.ams.org/msc/msc2010.html?t=&s=#1}{#1}}

\newtheorem{theorem}{Theorem}[section]
\newtheorem{definition}[theorem]{Definition}
\newtheorem{lemma}[theorem]{Lemma}

\newtheorem{corollary}[theorem]{Corollary}

\newtheorem{remark}[theorem]{Remark}

\newcommand{\R}{{\mathbb R}}
\newcommand{\N}{{\mathbb N}}
\newcommand{\Z}{{\mathbb Z}}
\newcommand{\C}{{\mathbb C}}

\newcommand{\bD}{{\mathbf D}}

\newcommand{\E}{\mathrm{e}}
\newcommand{\I}{\mathrm{i}}
\newcommand{\rE}{\mathrm{E}}

\newcommand{\loc}{\mathrm{loc}}

\newcommand{\sign}{\mathrm{sgn}}
\newcommand{\dom}{{\rm dom}}
\newcommand{\ran}{{\rm ran}}
\newcommand{\ac}{\mathrm{ac}}
\newcommand{\ess}{\mathrm{ess}}
\newcommand{\sing}{\mathrm{s}}

\newcommand{\be}{\begin{equation}}
\newcommand{\ee}{\end{equation}}

\newcommand{\wh}{\widehat}

\newcommand{\cF}{\mathcal{F}}

\newcommand{\cK}{\mathcal{K}}

\newcommand{\eps}{\varepsilon}

\newcommand{\cL}{\mathcal{L}}
\newcommand{\cH}{\mathcal{H}}
\newcommand{\cJ}{\mathcal{J}}
\newcommand{\cA}{\mathcal{A}}

\newcommand{\gH}{\mathfrak{H}}

\newcommand{\gq}{\mathfrak{q}}
\newcommand{\gt}{\mathfrak{t}}

\numberwithin{equation}{section}

\begin{document}

\title[Hamiltonian--Krein Index]{On the Hamiltonian--Krein index\\ for a non-self-adjoint spectral problem}

\author[A. Kostenko]{Aleksey Kostenko}
\address{Faculty of Mathematics and Physics\\ University of Ljubljana\\ Jadranska ul.\ 19\\ 1000 Ljubljana\\ Slovenia\\ and Faculty of Mathematics\\ University of Vienna\\ 
Oskar-Morgenstern-Platz 1\\ 1090 Wien\\ Austria\\ and RUDN University\\ Miklukho-Maklaya Str. 6\\ 
117198 Moscow\\Russia}
\email{\mailto{Aleksey.Kostenko@fmf.uni-lj.si};\ \mailto{Oleksiy.Kostenko@univie.ac.at}}
\urladdr{\url{http://www.mat.univie.ac.at/~kostenko/}}

\author[N. Nicolussi]{Noema Nicolussi}
\address{Faculty of Mathematics\\ University of Vienna\\
Oskar-Morgenstern-Platz 1\\ 1090 Wien\\ Austria}
\email{\mailto{nicolussin29@univie.ac.at}}

\thanks{{\it Research supported by the Austrian Science Fund (FWF) 
under Grant No.\ P28807 (A.K. and N.N.) and by the ``RUDN University Program 5-100" (A.K.).}}

\thanks{{\it Proc.\ Amer.\ Math.\ Soc., to appear}}


\begin{abstract}
We investigate the instability index of the spectral problem 
\[
-c^2y'' + b^2y + V(x)y = -\mathrm{i} z y'
\]
 on the line $\mathbb{R}$, where $V\in L^1_{\rm loc}(\mathbb{R})$ is real valued and $b,c>0$ are constants. This problem arises in the study of stability of solitons for certain nonlinear equations (e.g., the short pulse equation and the generalized Bullough--Dodd equation). We show how to apply the standard approach in the situation under consideration and as a result we provide a formula for the instability index in terms of certain spectral characteristics of the 1-D Schr\"odinger operator $H_V=-c^2\frac{d^2}{dx^2}+b^2 +V(x)$.
\end{abstract}

\maketitle

\section{Introduction}

Spectral problems of the form
\be\label{eq:1.01}
L u = z\, J u,\qquad z\in\C,
\ee
where $L=L^\ast$ is a lower semibounded self-adjoint operator and  $J=J^\ast=J^{-1}$ is a self-adjoint unitary operator in a Hilbert space $\gH$ naturally appear in the study of various important nonlinear equations (see, e.g.,  \cite{kp13}). It is well known that in the case if $L$ is a nonnegative operator, then in general the spectrum of \eqref{eq:1.01} is real  (for example, this holds if $0\in \rho(L)$). However, if $L$ has nonempty negative spectrum, then it turns out that \eqref{eq:1.01} might have nonreal eigenvalues and also real eigenvalues with Jordan chains of lengths more than $2$. The number of those eigenvalues is usually referred to as {\em the instability index} $\kappa_{\rm Ham}$ (a precise definition will be given below). The instability index plays a crucial role in the study of spectral and orbital stability of nonlinear waves and it turns out that it can be computed in terms of certain spectral characteristics of $L$. More precisely, avoiding some technical assumptions on the operators $L$ and $J$, this formula reads  
\be\label{eq:instindex}
\kappa_{\rm Ham} = \kappa_-(L) - \kappa_-(\bD),
\ee
where $\kappa_-(L)$ is the total multiplicity of the negative spectrum of $L$, $\kappa_-(L):=\dim P_{(-\infty,0)}(L)$, where $P_{\Omega}(L)$ is the spectral projection, and $\kappa_-(\bD)$ is the number of negative eigenvalues of {\em the constrained matrix} $\bD=(\bD_{k,n})$ usually defined by 
\be\label{eq:constrained}
\bD_{k,n} := \big(L^{-1}J\psi_k,J\psi_n\big)_{\gH},\qquad k,n=1,\dots,N,
\ee
and $\{\psi_k\}_{k=1}^N$ is an orthonormal basis in $\ker(L)$.  
These results were originally obtained by L.\ S.\ Pontryagin \cite{pon} and M.\ G.\ Krein \cite{kre1,kre2} (see also \cite{iokr}) in the 1940--1950s\footnote{Seems, in the finite dimensional case this goes back to the work of G.\ Frobenius \cite{fro}.} and then rediscovered later in connection with the study of stability problems for nonlinear waves (see \cite{cpv, gri, kks1,kks2,kp13, mac}, where further details and references can be found). 
In this theory it is essential that the operator $J$ in \eqref{eq:1.01} is bounded. However, certain nonlinear equations lead to spectral problems of the form \eqref{eq:1.01} with an unbounded operator $J$. For example (see Sections 4 and 5 in \cite{ss} for details),  the short pulse equation \cite{os} and the generalized Bullough--Dodd equation \cite{bd} lead to the spectral problem
\be\label{eq:indSpec}
- c^2y'' + b^2y + V(x) y =  - \I\, z\, y',\quad x\in\R,
\ee
where $c$, $b>0$ are fixed positive constants and  $V \in L_{\loc}^1(\R)$ is a real-valued function. 
This spectral problem has the form \eqref{eq:1.01}, however, instead of the operator $J$ on the right-hand side we have  
 $D \equiv  - \I d/dx$, which generates an unbounded operator on $L^2(\R)$.  Let $H_V$ be the maximal operator associated in $L^2(\R)$ with 
\begin{equation}\label{eq:l_V}
\tau_V = -c^2\, \frac{d^2}{dx^2} + b^2 + V(x).
\end{equation}
If $V\in L^1_{\loc}(\R)$ is real-valued and satisfies 
\begin{equation} \label{eq:compact}
\lim_{|x| \to \infty} \int_x^{x+1} |V(s)| \; ds = 0,
\end{equation}
then (cf. \cite{birman}) the operator $H_V$ is self-adjoint, bounded from below and its essential spectrum is $[b^2,\infty)$. M.\ Stanislavova and A.\ Stefanov \cite{ss} addressed the question whether \eqref{eq:indSpec} is {\em spectrally stable} in the sense of the following definition.

\begin{definition} \label{def:specInst}
A complex number $z\in\C$ is called {\em an eigenvalue} of \eqref{eq:indSpec} if there is $\psi_z\in \dom(H_V)$ called an eigenfunction, such that $\psi_z\neq 0$ and $H_V\psi_z = -\I z\psi_z'$. 
An eigenvalue $z\neq 0$ is called {\em unstable} if either $z\neq z^\ast$ or $z\in\R\setminus\{0\}$ and $(H_V\psi_z,\psi_z)_{L^2}\le 0$ for some eigenfunction $\psi_z\neq 0$. 
 
The spectral problem \eqref{eq:indSpec} is called {\em spectrally stable} if there are no unstable eigenvalues. Otherwise, it is called {\em spectrally unstable}.
\end{definition}

We use the asterisk to denote complex conjugation. It turns out that spectral instability is equivalent to the fact that the Hamiltonian--Krein index $\kappa_{\rm Ham}$ is positive. 
 In \cite{ss}, spectral stability of \eqref{eq:indSpec} was studied under the assumption that $H_V$ has exactly one negative eigenvalue (let us mention that in \cite{ss}, $H_V$ is not necessarily a 1-D Schr\"odinger operator, however, in applications to nonlinear equations it has exactly this form, see \cite{ss} for further details). Since the right-hand side in \eqref{eq:indSpec} gives rise to an unbounded operator in $L^2(\R)$, one needs to develop a new approach to investigate the instability index. In \cite{ss}, this was done by modifying the Evans function approach. Our main aim is to show how \eqref{eq:indSpec} can be reduced to the form \eqref{eq:1.01} with a bounded operator $J$ in order to then be able to apply the standard theory going back to the work of L.\ S.\ Pontryagin and M.\ G.\ Krein. Our approach has several advantages. First of all, it can be seen as a natural extension of the classical approach via the Krein space setting. Moreover, it enables us to compute the instability index in the case when $H_V$ has more than one negative eigenvalue. In particular, M.\ Stanislavova and A.\ Stefanov in \cite{ss} employed two different techniques for proving spectral stability resp. instability, whereas our approach covers both cases in a uniform manner. Moreover, we plan to develop it in a much wider setting in a forthcoming paper.  

Let us now formulate the results. Under the above assumptions on $V$, the kernel of the operator $H_V$ is at most one-dimensional since \eqref{eq:l_V} is limit point at $\infty$.  
If $\ker(H_V)={\rm span}\{\psi_0\}$ for some $\psi_0\neq 0$, we then set
\be\label{eq:bD_V}
\bD_V:= (H_V^{-1}\psi_0',\psi_0')_{L^2}.
\ee
We shall show in Section \ref{sec:III} that $\psi_0'\in \ran(H_V)$ and hence $\bD_V$ is defined correctly.

Our main result reads as follows.
 
\begin{theorem}\label{th:main}
Assume that $V\in L^1_{\loc}(\R)$ is real valued and satisfies \eqref{eq:compact}.
Let also $H_V$ be the maximal operator associated with \eqref{eq:l_V} in $L^2(\R)$ and $\kappa_-(H_V)\ge 1$ be the number of negative eigenvalues of $H_V$. 
 \begin{itemize} 
\item[{\rm (i)}] If  $\ker(H_V)=\{0\}$, then \eqref{eq:indSpec} is spectrally unstable. 
\item[{\rm (ii)}]  If $\ker(H_V)={\rm span}\{\psi_0\}\neq \{0\}$ and $\bD_V\neq 0$, then \eqref{eq:indSpec} is spectrally stable exactly when $\kappa_-(H_V)=1$ and $\bD_V<0$. 
\item[{\rm (iii)}] If either $\kappa_- (H_V)$ is odd and $\bD_V >0$ or  $\kappa_- (H_V)$ is even and $\bD_V <0$, then \eqref{eq:indSpec} has at least one purely imaginary eigenvalue $z \in \I \R_{>0}$. 
\end{itemize}
\end{theorem}

 \begin{remark}
Notice that in the case $\kappa_-(H_V) = 1$,  we  recover the results of M.\ Stanislavova and A.\ Stefanov, see \cite[Theorems 1-2]{ss}. 

Let us also mention that we touch upon the case $\bD_V = 0$  in Corollary \ref{cor:3.10}(iii). On the other hand, the analysis of stability of solitons becomes much more subtle in this case and we only refer in this respect to, e.g., \cite[Chapter 7]{kp13}.  
\end{remark}

Let us finish the introduction by briefly describing the content of the paper. Our main idea is to replace the Hilbert space $L^2(\R)$ by another Hilbert space $\dot{H}^{\frac12}(\R)$, which is defined as the completion of $L^2(\R)$ with respect to the norm 
\[
\|u\|_{\dot{H}^{\frac12}}^2:= \int_{\R} |\lambda| |\hat{u}(\lambda)|^2d\lambda.
\]
Here $\hat{u}:=\cF u$ denotes the Fourier transform of $u$ normalized by 
\be\label{eq:cF}
(\mathcal{F}u)(\lambda) = \hat{u}(\lambda) = \frac{1}{\sqrt{2\pi}} \int_{\R} \E^{-\I \lambda x}u(x)\,dx,\quad u\in L^1(\R)\cap L^2(\R).
\ee
In Section \ref{sec:II}, using the form approach, we develop the spectral theory of the operator $\cL$ defined in $\dot{H}^{\frac12}(\R)$ by the expression 
\[
c^2|D| + b^2|D|^{-1} + |D|^{-1}V(x),
\]  
where  
\be\label{eq:|D|}
|D|^a\colon u \mapsto \mathcal{F}^{-1} (|\lambda|^a \hat{u}(\lambda)).
\ee
for every $a\in\R$. In Section \ref{sec:III}, we show that the spectral problem \eqref{eq:indSpec} if considered in $\dot{H}^{\frac12}(\R)$ has the following form 
\be\label{eq:spnew}
\cL f= z\, \cJ f,
\ee
where $\cJ=\I\cH$ and $\cH$ is the Hilbert transform, which is a skew-self-adjoint and unitary operator in $\dot{H}^{\frac12}(\R)$. Therefore, we can apply the standard formula \eqref{eq:instindex} to compute the instability index of \eqref{eq:spnew}. The final step in the proof of our main result is the proof of the fact that the point spectrum  (including algebraic and geometric multiplicities) of the new spectral problem \eqref{eq:spnew} coincides with the point spectrum of the original problem \eqref{eq:indSpec} considered in $L^2(\R)$. 

Finally, in Appendix \ref{app:01} we collect basic notions and facts on quadratic forms.

\section{The auxiliary self-adjoint spectral problem}\label{sec:II}

The main focus of this section is on the auxiliary spectral problem 
\be\label{eq:Spec}
\tau_V(u) = z|D|u,
\ee
where $\tau_V$ is the differential expression \eqref{eq:l_V} and $|D|$ is defined by \eqref{eq:|D|}. 
In contrast to \cite{ss}, we are going to consider \eqref{eq:indSpec} and \eqref{eq:Spec} in the Hilbert space $\dot{H}^{\frac12}(\R)$. More precisely, let $H^s(\R)$, $s\in\R$ be the standard scale of Sobolev spaces. In particular,
\be\label{eq:H12}
\begin{split}
 H^{1/2}(\R):=&\{u \in L^2(\R)\colon |D|^{1/2} u \in L^2(\R)\}\\
 =&\{u \in L^2(\R)\colon |\lambda|^{1/2} \hat{u}(\lambda) \in L^2(\R)\}.
 \end{split}
\ee
Denote by $\dot{H}^{\frac12}(\R)$ the closure of $H^{1/2}(\R)$ with respect to the norm
\begin{equation}\label{eq:dotH}
\|u\|_{\dot{H}^{\frac12}} := \||D|^{1/2}u\|_{L^2}=\| |\lambda|^{1/2} \hat u \|_{L^2}. 
\end{equation}
Notice that 
\begin{equation}
\|u\|_{\dot{H}^{\frac12}}^2 = (|D|u,u)_{L^2} \le \|u\|_{H^1} \|u\|_{L^2},
\end{equation}
whenever  $u\in H^1(\R)$. Clearly, $H^1(\R)$ is dense in $\dot{H}^{\frac12}(\R)$. Moreover, $\dot{H}^{\frac12}(\R)$ is isometrically isomorphic to $L^2(\R)$. We write $j$ for the particular isometric isomorphism obtained by continuously extending
\begin{equation}\label{eq:j0}
\begin{array}{cccc}	
	j_0\colon & H^{1/2}(\R) &  \to & L^2(\R),\\[1mm]
	                & u   & \mapsto  & |\lambda|^{1/2} \hat u (\lambda).
\end{array}	                
\end{equation}

\subsection{The unperturbed case: $V\equiv 0$}\label{sec:II.1}

Assume that $b$, $c>0$. Let $\cL_0$ be the operator associated in $\dot{H}^{\frac12}(\R)$ with the spectral problem
\be\label{eq:SpecUnp}
-c^2 y'' + b^2y = z |D|y.
\ee
More precisely, 
\be
\begin{array}{cccc}
\cL_0 \colon & \dom(\cL_0) & \to & \dot{H}^{\frac12}(\R), \\[1mm]
  & u & \mapsto & c^2 |D|u + b^2|D|^{-1}u,
\end{array}
\ee
where $\dom(\cL_0)$ is the maximal domain,
\be 
 \dom(\cL_0) = \{u\in \dot{H}^{\frac12}(\R)\colon \, |D|u,\ |D|^{-1}u\in \dot{H}^{\frac12}(\R)\}.
\ee
The spectral properties of $\cL_0$ can easily be described by using the Fourier transform.

\begin{lemma}\label{lem:2.01}
The operator $\cL_0$ is self-adjoint and its spectrum is purely absolutely continuous,
\[
\sigma(\cL_0) = \sigma_{\ac}(\cL_0) = [2bc,\infty),\quad \sigma_{\sing}(\cL_0) = \emptyset.
\]
\end{lemma}

\begin{proof}
Denote by $\wh{\mathcal{L}}_0$ the self-adjoint multiplication operator on $L^2(\R)$ given by

\begin{equation}
\begin{array}{cccc}
\wh{\mathcal{L}}_0 \colon &  \dom(\wh{\mathcal{L}}_0) & \to &  L^2(\R),\\
& u &\mapsto& \left(c^2 |\lambda| + b^2 |\lambda|^{-1}\right)u,
\end{array}
\end{equation}
on the maximal domain 
\[
\dom(\wh{\mathcal{L}}_0) = \{u \in L^2(\R)\colon \, (|\lambda| + |\lambda|^{-1})u \in L^2(\R)\}.
\]
It is not difficult to show that the operators $\cL_0$ and $\wh\cL_0$ are unitarily equivalent and 
\be\label{eq:domL0}
\mathcal{L}_0 = j^{-1} \wh{\mathcal{L}}_0 j.
\ee  
Now the claim follows from the spectral properties of $\wh\cL_0$.
\end{proof}

\begin{remark}
By  \eqref{eq:domL0}, $\dom(\mathcal{L}_0)= j^{-1}(\dom(\wh{\mathcal{L}}_0))$  and hence \eqref{eq:j0} implies 
\be
\dom({\cL}_0) = \{u\in\dot{H}^{\frac12}(\R)\colon |\lambda|^{3/2}\hat{u},\ |\lambda|^{-1/2}\hat{u}\in L^2(\R)\}.
\ee
Therefore, $\dom({\cL}_0)\subset L^2(\R)$ since 
\begin{align*}
\|u\|^2_{L^2}=\int_{\R}|\hat{u}(\lambda)|^2 d\lambda &\le \int_{|\lambda|<1}\frac{1}{|\lambda|}|\hat{u}(\lambda)|^2 d\lambda + \int_{|\lambda|>1}{|\lambda|^3}|\hat{u}(\lambda)|^2 d\lambda\\
&\le \||\lambda|^{-1/2}\hat{u}\|^2_{L^2} + \||\lambda|^{3/2}\hat{u}\|^2_{L^2} \\
&= \||D|^{-1} u\|^2_{\dot{H}^{\frac12}(\R)} + \||D| u\|^2_{\dot{H}^{\frac12}(\R)}.
\end{align*}
The latter also implies $\dom(\cL_0)\subseteq H^{3/2}(\R)$.
\end{remark}

Let us consider the following quadratic form  in $\dot{H}^{\frac12}(\R)$
\be\label{eq:formL0}
\gt_0[u] := c^2 \int_{\R}|u'(x)|^2\, dx + b^2\int_{\R}|u(x)|^2\, dx, 
\ee
defined on the maximal domain
\be\label{eq:domt0}
\dom(\gt_0):= \{u\in \dot{H}^{\frac12}(\R)\colon \gt_0[u]<\infty\}.
\ee
Since $H^1(\R)\subset \dot{H}^{\frac12}(\R)$, we get $\dom(\gt_0)=H^1(\R)$. Clearly, the form $\gt_0$ is closed in $\dot{H}^{\frac12}(\R)$ (this can be seen by applying the isometry $j$).  Moreover, 
\[
(\cL_0 u,u)_{\dot{H}^{\frac12}} = \gt_0[u]
\]
for all $u\in\dom(\cL_0)$ and hence $\cL_0$ is the self-adjoint operator associated with the form $\gt_0$ in $\dot{H}^{\frac12}(\R)$. This implies that $\dom(\cL_0)\subset H^1(\R)$ and $\dom(\sqrt{\cL_0}) = H^1(\R)$.

\subsection{The case $V\not\equiv 0$} \label{sec:II.2}

Let $V \in L_{\loc}^1(\R)$ be a real-valued function satisfying
\begin{equation} \label{eq:cond_int_bd}
		M_V := \sup_{n \in \Z} \; \int_{n}^{n+1} |V(x)| \; dx < \infty.
\end{equation}
Our main aim is to associate a self-adjoint operator $\cL$ acting in $\dot{H}^{\frac12}(\R)$ with the spectral problem \eqref{eq:Spec}. Our main tool in dealing with \eqref{eq:Spec} is the form approach.

Here and below we shall use the following notation
\be\label{eq:sqrtV}
V^{1/2}:= \frac{V}{|V|^{1/2}}.
\ee
It is well known (see, e.g., \cite{kato}) that under these conditions $V^{1/2} u \in L^2(\R)$ for every $u \in H^1(\R)$ and
\begin{equation} \label{eq:est_multiplicationbysqrt}
			\| V^{1/2} u \|_{L^2}^2 \leq M_V\big(\eps \|u'\|_{L^2}^2 + (1 + \eps^{-1}) \|u\|_{L^2}^2\big)
\end{equation}
where $\eps >0$ is arbitrary. Hence the quadratic form defined by
\begin{equation}\label{eq:gq_V}
			\gq_V[u] := (V^{1/2} u, |V|^{1/2} u )_{L^2} = \int_{\R} |u(x)|^2V(x)\,dx,
\end{equation}
for all $u \in H^1(\R)$ is form bounded with respect to the form $\gt_0$ given by \eqref{eq:formL0}. 

Writing the spectral problem \eqref{eq:Spec} in the form
\be\label{eq:Spec02}
(c^2 |D| + b^2 |D|^{-1} + |D|^{-1} V) u = z\, u,
\ee
we want to interpret it as a perturbation of \eqref{eq:SpecUnp} by the term $ |D|^{-1} V$. 
By formal computation (or at least for suitable $u$),
\begin{equation}
\gq_V[u]=(|D|^{-1}Vu, u)_{\dot{H}^{\frac12}}, 
\end{equation}
which hints that the form $\gq_V$ considered as a form in the Hilbert space $\dot{H}^{\frac12}(\R)$ should be used to represent this additional term. More precisely, let us consider in $\dot{H}^{\frac12}(\R)$ the following form
\be\label{eq:formL}
\gt_V[u]:= \gt_0[u] + \gq_V[u],\quad \dom(\gt_V):=\dom(\gt_0)=H^1(\R).
\ee
The form $\gq_V$ considered as a form in $L^2(\R)$ is infinitesimally form bounded with respect to the free Hamiltonian $H_0 := -\frac{d^2}{dx^2}$ if $V$ satisfies \eqref{eq:cond_int_bd} (see \eqref{eq:est_multiplicationbysqrt}) and hence by the KLMN Theorem \ref{th_KLMN}, the quadratic form $\gt_V$ considered as a form in $L^2(\R)$ is bounded from below and closed. However, it seems the form $\gt_V$ considered as a form in the Hilbert space $\dot{H}^{\frac12}(\R)$ may not be strongly $\gt_0$-bounded if $V$ satisfies \eqref{eq:cond_int_bd} (the $\dot{H}^{\frac12}$ norm does not control the $L^2$ norm at small energies). Therefore, in order to define the form $\gt_V$ in $\dot{H}^{\frac12}(\R)$ we need the following extra condition
\begin{equation} \label{eq:cond_int_limit_0}
\lim_{|x| \to \infty} \int_x^{x+1} |V(s)| \; ds = 0.
\end{equation}

\begin{lemma} \label{lem:relcomp_birman}
	Let $V \in L_\loc^1(\R)$ be a real-valued function satisfying \eqref{eq:cond_int_limit_0}. Then:
 \begin{itemize}
\item[{\rm (i)}]  The form $\gq_V$ is infinitesimally form bounded with respect to the form $\gt_0$.
\item[{\rm (ii)}]  The form $\gt_V$ is closed and lower semibounded.
\item[{\rm (iii)}]  If $\mathcal{L}:=\cL_0+|D|^{-1}V$ is the self-adjoint operator on $\dot{H}^{\frac12}(\R)$ associated with the form  $\gt_V$, then the resolvent difference of $\cL$ and $\cL_0$ is compact for every $z \in \rho(\cL_0)\cap \rho(\cL)$ and hence their essential spectra coincide,
	\[
	\sigma_{\ess}(\cL) = \sigma_{\ess}(\cL_0) = [2bc,\infty).
	\]
\end{itemize}
\end{lemma}

\begin{proof}
	(i) First of all, by the Cauchy--Schwarz inequality we get
	\begin{equation*}
		\left( \int_{\R} |\hat{u}(\lambda)| \; d\lambda \right)^2  \leq \left ( \int_{\R} \frac{d \lambda}{1 + |\lambda|^2 + 2r |\lambda|}   \right) (\| u \|_{H^1}^2 + 2r \|u\|_{\dot{H}^{\frac12}})
	\end{equation*}
	for every $u \in H^1(\R)$ and $r>0$. 	Moreover,
\begin{align*}
	\int_\R \frac{d \lambda}{1 + |\lambda|^2 + 2r |\lambda|} &
	=2\int_{r}^\infty \frac{d \lambda}{\lambda^2 -(r^{2}-1)} 
	   = \frac{1}{\sqrt{r^{2}-1}} \log\left(\frac{r + \sqrt{r^{2} - 1}}{r - \sqrt{r^{2}-1}} \right)\\
	    =  &\frac{2}{\sqrt{r^{2}-1}} \log(r + \sqrt{r^2-1}) \le \frac{2\log(2r)}{\sqrt{r^{2}-1}} \le \frac{4 \log(2r)}{r}
\end{align*}
	for every $r\ge 2/\sqrt{3}$. This means that the estimate
	\begin{equation}\label{eq:Sobolev_dot}
			\sup_{x \in \R} \; |u(x)|^2 \leq \frac{4 \log(2r)}{r} \|u\|_{H^1}^2 + 8\log(2r)\|u\|^2_{\dot{H}^{\frac12}}
	\end{equation}
holds  for all $u\in H^1(\R)$ whenever $r \ge 2/\sqrt{3}$.	Now using inequality (IV.1.19) from \cite{kato}
\begin{equation} \label{eq:ineq_sup_h1}
  	\sup_{x \in [a,b]}  \; |v(x)| \; \leq \; \sqrt{\frac{(b-a)}{2m+3}}\|v'\|_{L^2} + \frac{m+1}{\sqrt{(b-a)(2m+1)}}\|v\|_{L^2},  
\end{equation}
which holds for all $v \in H^1([a,b])$ and $m\in \R_{>0}$, we arrive at the estimate 
	\begin{align*}
		|\gq_V[u]| \leq & \sum_{|n|\geq N} \int_n^{n+1} |V(s)| |u(s)|^2 \; ds + \int_{-N+1}^{N} |V(s)| |u(s)|^2 \; ds \\
		\leq &   \|u\|_{H^1(\R)}^2 \sup_{|x|\geq N} \int_x^{x+1} |V(s)|  \; ds \\
		&+  \left ( \frac{4 \log(2r)}{r} \|u\|_{H^1(\R)}^2 + 8\log(2r) \| u \|_{\dot{H}^{\frac12}(\R)}^2 \right) \int_{-N+1}^{N} |V(s)| \; ds
	\end{align*}
	for every $N \in \N$ and $r \ge 2/\sqrt{3}$. Using \eqref{eq:cond_int_limit_0} and noting that $\lim_{r\to \infty} \frac{\log(2r)}{r} =0$, we conclude that for every $\delta >0$ there is a constant $K_\delta>0$ such that 
	\begin{equation*}
		| \gq_V[u]| \leq \delta \|u\|_{H^1}^2 + K_\delta \|u\|_{\dot{H}^{\frac12}}^2
	\end{equation*}
holds for every $u\in \dom(\gt_0)=H^1(\R)$. Hence the form $\gq_V$ is infinitesimally form bounded with 
respect to $\gt_0$. 
	
(ii) Follows from (i) and the KLMN Theorem \ref{th_KLMN}. 

(iii) By (ii) and the first representation theorem \cite[Chapter VI.2.1]{kato}, the form $\gt_V$ gives rise to a semi-bounded, self-adjoint operator $\cL=\cL_0+|D|^{-1}V$ (defined as a form sum) on $\dot{H}^{\frac12}(\R)$ with $\dom(\cL)\subset H^1(\R)$ and $\dom(\cL^{1/2}) = H^1(\R)$. By \cite[Theorem 2.1]{birman}, $H^1(\R)$ is compactly embedded into $L^2(\R;|V|)$ if and only if $V$ satisfies \eqref{eq:cond_int_limit_0}. Hence, by Lemma \ref{lem3.1}, the form $\gq_V$ is compact on $\dom(\cL_0^{1/2})$ equipped with the graph norm if  $V$ satisfies \eqref{eq:cond_int_limit_0} and it remains to apply  Birman's Theorem \ref{th2.2}.
\end{proof}

\begin{remark}
It is known that the form $\gq_V$ given by \eqref{eq:gq_V} and considered as a form in $L^2(\R)$ is infinitesimally form bounded with respect to the free Hamiltonian $H_0 := -\frac{d^2}{dx^2}$ if $V$ satisfies \eqref{eq:cond_int_bd}. However, in the space $\dot{H}^{\frac12}$ we were able to prove that $\gq_V$ is $\gt_0$-infinitesimally form bounded under the additional assumption \eqref{eq:cond_int_limit_0}. On the other hand, as in the case of $L^2(\R)$, condition \eqref{eq:cond_int_limit_0} is necessary and sufficient for the form $\gq_{|V|}$ to be relatively compact with respect to $\gt_0$.
\end{remark}

\subsection{A bound on the number of negative eigenvalues}\label{sec:II.3}
Assume that $V \in L^1_\loc(\R)$ is a real-valued function satisfying \eqref{eq:cond_int_limit_0}. By Lemma \ref{lem:relcomp_birman}, the spectrum of $\cL$ in $(-\infty, 2bc)$ consists of eigenvalues which may accumulate only at $E_0=2bc$. Our next aim is to derive a bound on the number of negative eigenvalues of $\cL$. We denote by $\kappa_-(T)$ the total multiplicity of the negative spectrum of a self-adjoint operator $T=T^\ast$, $\kappa_-(T) := \dim P_{(-\infty,0)}(T)$, where $P_{\Omega}(T)$ is the spectral projection. 

We begin with an estimate which follows from the classical Bargmann bound for 1-D Schr\"odinger operators on the line.

\begin{lemma}\label{lem:bargmann}
Let $V\in L^1_{\loc}(\R)$ satisfy \eqref{eq:cond_int_limit_0} and let $\cL$ be the corresponding self-adjoint operator in $\dot{H}^{\frac12}(\R)$. Then 
\be\label{eq:kappa-}
\kappa_-(\cL) = \kappa_-(H_V),
\ee
where $H_V$ is the Sch\"odinger operator \eqref{eq:l_V} defined on the maximal domain in $L^2(\R)$. 

If $V(x)\ge -b^2$ for a.e. $x\in\R$, then $\cL\ge 0$. Otherwise, 
\be\label{eq:bargmann}
\kappa_-(\cL) \le 1+\frac{1}{c^2}\int_{\R} |x| |\min\{V(x)+b^2,0\}| dx.
\ee 
\end{lemma}

\begin{proof}
As it was already mentioned, the form $\gt_V$ if considered in $L^2(\R)$ instead of $\dot{H}^{\frac12}(\R)$ coincides with the form of the Schr\"odinger operator $H_V$ defined by \eqref{eq:l_V} on the maximal domain. Therefore, the minimax principle implies \eqref{eq:kappa-}. 

Clearly, the form $\gt_V$ given by \eqref{eq:formL} is nonnegative if so is the potential $V+b^2$ and hence $\cL$ is nonnegative in this case. 
Applying the classical Bargmann estimate to $H_V$ (see \cite[Theorem 7.5]{simon_ti}), we end up with \eqref{eq:bargmann}.
\end{proof}

The next bound again follows from the standard estimates for 1-D Schr\"odinger operators. 

\begin{lemma}\label{lem:2.9}
Assume $V\in L^1(\R)$ is real valued and set $V_-:=(V-|V|)/2$. 
Then
\be\label{eq:BSestimate}
	\dim P_{(-\infty,0]}(\cL) \leq  \frac{1}{2bc} \int_{\R} |V_-(y)| \; dy.	
\ee
\end{lemma}
 
\begin{proof}
Notice that $\kappa_-(H_V)$ is equal to the number of eigenvalues of the 1-D Schr\"o\-din\-ger operator $-\frac{d^2}{dx^2} + \frac{1}{c^2}V(x)$ lying below $-(b/c)^2$. Since the resolvent of the free Hamiltonian $H_0$ is given by
\[
(H_0+\lambda^2)^{-1} g = \frac{1}{2\lambda}\int_{\R} \E^{-\lambda|x-y|}g(y)dy,\quad \lambda>0,
\]
and then using the Birman--Schwinger principle (see, e.g., \cite[Chapter VII]{simon_ti}), it is not difficult to show that 
\[
\dim P_{(-\infty,0]}(H_V) \leq  \frac{1}{2bc} \int_{\R} |V_-(y)| \; dy.
\]
It remains to use \eqref{eq:kappa-} and the fact that $\ker(\cL) = \ker(H_V)$ (see Corollary \ref{cor:ker=ker}). 
\end{proof} 

\begin{remark}
It is possible to apply the Birman--Schwinger approach to the operator $\cL$ and then, for example, to investigate the number of eigenvalues of $\cL$ lying below the threshold $E_0=2bc$. However, these results are not needed for our purposes and hence we do not touch this issue here.
\end{remark}

\section{The indefinite spectral problem}\label{sec:III}

\subsection{The unperturbed case} \label{sec:III.1}
Our main aim is to investigate spectral properties of the problem \eqref{eq:indSpec}. In contrast to \cite{ss}, we are going to consider it in the space $\dot{H}^{\frac12}(\R)$. As in the previous section, we begin with the unperturbed case $V\equiv 0$,
\be
-c^2 y'' + b^2y = -\I\,z\, y'.
\ee
First, we define the operator  $\cA_0$  in $\dot{H}^{\frac12}(\R)$ as follows
\be
\begin{array}{cccc}
\cA_0 \colon & \dom(\cA_0) & \to & \dot{H}^{\frac12}(\R), \\[1mm]
  & u & \mapsto & c^2 Du + b^2D^{-1}u,
\end{array}
\ee
where $\dom(\cA_0)$ is the maximal domain,
\be 
 \dom(\cA_0) = \{u\in \dot{H}^{\frac12}(\R)\colon Du,\ D^{-1}u\in \dot{H}^{\frac12}(\R)\}.
\ee
Here $D^n$, $n\in\Z$ is defined via the Fourier transform $\mathcal{F}$ by
\be\label{eq:D}
D^n\colon u \mapsto \mathcal{F}^{-1} (\lambda^n \hat{u}(\lambda)).
\ee
Clearly, for positive $n\in\N$, $D^n u = (-\I)^n u^{(n)}$ for all suitable functions $u$. Notice also that $\dom(\cA_0)=\dom(\cL_0)$. As in the case of the operator $\cL_0$, spectral properties of $\cA_0$ can easily be described by using the Fourier transform.

\begin{lemma}\label{lem:3.01}
The operator $\cA_0$ is self-adjoint and 
\[
\sigma(\cA_0) = \sigma_{\ac}(\cA_0) = (-\infty,-2bc]\cup [2bc,\infty),\quad \sigma_{\sing}(\cA_0) = \emptyset.
\]
\end{lemma}

\begin{proof}
Define the operator $\cJ$ on $\dot{H}^{\frac12}$ by 
\be\label{eq:cJ}
\begin{array}{cccc}
\cJ \colon & \dot{H}^{\frac12} & \to & \dot{H}^{\frac12} \\
 & u & \mapsto & \cF^{-1}(\sign(\lambda)\hat{u}).
\end{array}
\ee
Clearly, $\cJ = \I \cH$, where $\cH$ is the Hilbert transform and $\cJ=\cJ^\ast =\cJ^{-1}$. Moreover, it commutes with $\cL_0$ and $\cA_0=\cJ\cL_0$.
This and Lemma \ref{lem:2.01} complete the proof.
\end{proof}

\subsection{The case $V\not\equiv 0$}\label{sec:III.2}

The operator $\cJ$ induces a Krein space structure on $\dot{H}^{\frac12}(\R)$. Namely, consider the inner product
\be
\langle f,g\rangle_{\dot{H}^{\frac12}}:= (\cJ f,g)_{\dot{H}^{\frac12}} = \int_{\R}\hat{f}(\lambda)\hat{g}(\lambda)^\ast\, \lambda\,d\lambda,\quad f,g \in \dot{H}^{\frac12}(\R).
\ee
Then  $\cK:=(\dot{H}^{\frac12}(\R),\langle \cdot,\cdot\rangle_{\dot{H}^{\frac12}})$ is a Krein space (see \cite{azio, bognar,lan}). Now we set 
\be\label{eq:cA}
\cA:= \cJ\cL,
\ee
where $\cL$ is the operator introduced in Section \ref{sec:II.2}. Since $\cL=\cL^*$, the operator $\cA$ is self-adjoint in the Krein space $\cK$.

\begin{lemma}\label{lem:defin}
Assume that $V\in L^1_{\loc}(\R)$ is real valued and satisfies \eqref{eq:cond_int_limit_0}. Then the operator $\cA$ is {\em definitizable}, that is, $\rho(\cA)\neq \emptyset$ and there is a real polynomial $p$, called {\em a definitizing polynomial}, such that 
\be\label{eq:definitiz}
\big\langle\,p(\cA)f,f\,\big\rangle_{\dot{H}^{\frac12}} \ge 0
\ee
for all $f\in \dom(\cA^{\deg(p)})$.
\end{lemma} 

\begin{proof}
By Lemma \ref{lem:relcomp_birman}(iii), $\kappa_-(\cL)<\infty$ if \eqref{eq:cond_int_limit_0} holds true. If $0\in \rho(\cL)$, then $0\in \rho(\cA)$ and hence the claim follows (see item (c) on p.12 in \cite{lan}). 

If $0\in \sigma(\cL)$, then,  by Lemma \ref{lem:relcomp_birman}(iii), $0$ is an isolated eigenvalue of $\cL$. Let $P_0:=P_0(\cL)$ be the orthogonal projection onto $\ker(\cL)$. Then the operator 
\[
\cA_\varepsilon:= \cJ(\cL+\varepsilon P_0) = \cA+\varepsilon \cJ P_0,\qquad \cL_\varepsilon:=\cL+\varepsilon P_0,
\]
is definitizable for all $\varepsilon>0$ since $\kappa_-(\cL_\varepsilon) = \kappa_-(\cL)<\infty$ and $0\in \rho(\cL_\varepsilon)$. It remains to note that $\cA_\varepsilon$ is a rank one perturbation of $\cA$ since $\dim\ker(\cL)=1$ and then apply \cite[Theorem 1]{jl}.
\end{proof}

It follows from the proof of item (c) on p.12 in \cite{lan}, that $p$ can be chosen such that $p(z) = zp_0(z)(p_0(z^\ast))^\ast$ with $\deg(p_0)\le \kappa_-(\cL)$. For further details on spectral theory of definitizable operators we refer to \cite{lan}. Let us only mention the following important properties of the point spectrum of $\cA$ (cf. \cite[Proposition 2.1]{lan}). If $p$ is a fixed definitizing polynomial for $\cA$, then let us denote by $k(z)$ the multiplicity of $z$ as a zero of $p$ (in particular, $k(z)=0$ if $p(z)\neq 0$). 

\begin{corollary}\label{cor:3.3}
The spectrum of a definitizable operator $\cA$ is symmetric with respect to the real axis and the nonreal spectrum of $\cA$ consists of isolated eigenvalues of total algebraic multiplicity at most $2\kappa_-(\cL)$. Moreover, every isolated eigenvalue $z$ of $\cA$ has finite Riesz index $\nu(z)$ and 
\be
\nu(z) \le \begin{cases} k(z), & z\in\C\setminus\R,\\ k(z)+1, & z\in\R. \end{cases}
\ee
\end{corollary}

\subsection{The Hamiltonian--Krein index of the operator $\cA$} Let $\kappa_{\C_+}(\cA)$ be the total algebraic multiplicity of eigenvalues of $\cA$ lying in the open upper half-plane $\C_+$.  As an immediate corollary of Lemma \ref{lem:defin} we arrive at the following estimate
\be
\kappa_{\C_+}(\cA)  \le \kappa_-(\cL). 
\ee
If $\lambda\in \R\setminus\{0\}$ is an eigenvalue of $\cA$, then we shall denote by 
$\kappa_{\lambda}^-(\cA)$ its negative index. More precisely, if $\rE_\lambda(\cA)$ is the generalized eigenspace of $\cA$ corresponding to $\lambda$, 
 then $\kappa_\lambda^-(\cA)$ is the number of non-positive eigenvalues of the operator $\cL$ restricted to $\rE_\lambda(\cA)$.
It can be shown that  $\kappa_\lambda^-(\cA)$ is in fact equal to the number of negative eigenvalues of $\cL_\lambda := P_\lambda \cL P_\lambda$ if $\lambda$ is a {\em normal} eigenvalue of $\cA$, that is, if $\lambda$ is an isolated eigenvalue, $\dim(\rE_\lambda)<\infty$ and $\cA\upharpoonright \rE_\lambda^\perp$ is boundedly invertible (see \cite[Corollary VI.6.6]{bognar}). Here $P_\lambda$ denotes the corresponding orthogonal  projection in $\dot{H}^{\frac12}(\R)$ onto $\rE_\lambda(\cA)$. 
The total negative Krein index is then defined by
\be
\kappa_{\R}^-(\cA) := \sum_{\lambda\in \R\setminus\{0\}} \kappa_\lambda^-(\cA).
\ee
The number 
\be
\kappa_{\rm Ham}(\cA):= \kappa_{\C_+}(\cA) + \kappa_{\R}^-(\cA)
\ee
is called {\em the Hamiltonian--Krein index} of $\cA$. 

 If $\ker(\cL)\neq \{0\}$, then we need to introduce the constrained matrix $\bD$. Noting that $\ker(\cL)=\ker(H_V)$ is at most one-dimensional, we conclude that $\ker(\cL) = {\rm span}\{\psi_0\}$ with some $\psi_0\neq 0$. Now we set
\be\label{eq:bD}
\bD= (\cL^{-1}\cJ \psi_0,\cJ\psi_0)_{\dot{H}^{\frac12}},
\ee
and then we define the following quantity 
\be
\kappa_-(\bD) = \begin{cases} 0, & \bD > 0,\\ 1, & \bD< 0.  \end{cases}
\ee
Notice that \eqref{eq:bD} is well defined. Indeed, if $0\in\sigma(\cL)$, then $0$ is an isolated eigenvalue. Hence $\cJ\psi_0\in \ran(\cL)$ only if $\cJ\psi_0\perp \ker(\cL)$. However, 
\[
(\cJ\psi_0,\psi_0)_{\dot{H}^{\frac12}} = (\I|D|\cH\psi_0,\psi_0)_{L^2} = (D\psi_0,\psi_0)_{L^2}=-\I(\psi_0',\psi_0)_{L^2} = 0,
\]
since $\psi_0\in H^1(\R)$ is real valued (up to a scalar multiple). 

\begin{lemma}\label{lem:kappaA}
Assume that $V\in L^1_{\loc}(\R)$ is real valued and satisfies \eqref{eq:cond_int_limit_0}. 
\begin{itemize}
\item[{\rm (i)}] If $\ker(\cL)=\{0\}$, then $\kappa_{\rm Ham}(\cA) = \kappa_-(\cL)$. 
\item[{\rm (ii)}] If $\ker(H_V)\neq\{0\}$ and $\bD\neq 0$, then
\be\label{eq:kappaA1}
\kappa_{\rm Ham}(\cA) = \kappa_-(\cL) - \kappa_-(\bD).
\ee
\end{itemize}
\end{lemma}

\begin{proof}
The proof immediately follows from the Hamiltonian--Krein index theorem, see \cite[Theorem 7.1.5]{kp13}.
\end{proof}

Notice that several estimates on $\kappa_-(\cL)$ in terms of $V$ are given in the previous section, see Lemma \ref{lem:bargmann} and Lemma \ref{lem:2.9}.

\subsection{Proof of Theorem \ref{th:main}}
Our aim is to show that the eigenvalues (counting their multiplicities) of \eqref{eq:indSpec} and $\cA$ coincide. Recall that $z\in \C$ is called an eigenvalue of \eqref{eq:indSpec} if there is  $u\in \dom(H_V)$ such that $H_V u= - i z  u'$. Since $\dom(H_V)\subset \dom(\gt_0) = H^1(\R)$, we conclude that $z$ is an eigenvalue of \eqref{eq:indSpec} if there is  $u\in \dom(H_V)$ such that $u\neq 0$ and $H_Vu=zDu$. 
On the other hand, since $\cJ=\cJ^{-1}$, $z\in\C$ is an eigenvalue of $\cA$ if there is $u\in \dom(\cL)$ such that $u\neq 0$ and  $\cL u = z \cJ u$. 

The generalized eigenspaces for $\cA$ are defined in a standard way:
\be\label{eq:genE_A}
\rE_z(\cA):= \bigcup_{n\in\N} \ker\big((\cA-z)^n\big).
\ee
The generalized eigenspaces for \eqref{eq:indSpec} can be formally written as
\be\label{eq:genE_A}
\rE_z(A):= \bigcup_{n\in\N} \ker\big((D^{-1}H_V-z)^n\big).
\ee
In order to define $\rE_z(A)$ properly, we set
\[
\ker(D^{-1}H_V-z):= \ker(H_V-zD) = \{u\in\dom(H_V)\colon H_Vu = zDu\},
\]
and then we define $\ker((D^{-1}H_V-z)^{n+1})$ for all $n\ge 1$ as follows
\be\label{eq:kerHn}
\ker\big((D^{-1}H_V-z)^{n+1}\big):=\big\{u\in\dom(H_V)\colon H_Vu-zDu \in D(\ker\big((D^{-1}H_V-z)^n)\big)\big\}.
\ee

\begin{lemma} \label{lem:gev}
Let $V \in L^1_\loc(\R)$ be a real-valued function satisfying \eqref{eq:cond_int_limit_0}. Then 
\begin{equation}\label{eq:gev1}
		\ker(H_V-zD)	= \ker (\cL-z\cJ)
\end{equation}
for every $z\in\C$. Moreover,  
	\begin{equation} \label{eq:gev}
		\ker \left((D^{-1}H_V -z)^n \right) = \ker \left(( \cA -z)^n \right),\quad n \in \N,
	\end{equation}
 and the corresponding generalized eigenspaces $\rE_z(A)$ and $\rE_z(\cA)$ coincide. 
\end{lemma}

\begin{proof}
Let us first prove \eqref{eq:gev1}.  Since $H_V$ and $\cL$ are the self-adjoint operators associated with the quadratic form $\gt_V$ on $L^2(\R)$ resp. $\dot{H}^{\frac12}(\R)$, the first representation theorem (see \cite[Theorem VI.2.1]{kato}) yields
	\begin{align} \label{eq:repH}
		\dom(H_V) &= \{u \in H^1(\R) \colon  \exists w \in L^2(\R) \ \text{such that}\ \gt_V[u,v] = (w, v)_{L^2} \; \forall v \in H^1(\R) \},
	\end{align}
with $H_Vu = w$, and
\begin{align} \label{eq:repL}
		\dom(\cL) &= \{{u} \in H^1(\R) \colon  \exists \tilde w \in \dot{H}^{\frac12}(\R) \ \text{such that}\ \gt_V[ u,  v] = (\tilde w,  v)_{\dot{H}^{\frac12}} \; \forall v \in H^1(\R) \},
\end{align}
with $\cL  u = \tilde w$. 

If $u \in \ker(H_V -zD)$, then $u\in\dom(H_V)$ and $H_Vu=zDu$. Hence $D u = w\in L^2(\R)$ and moreover applying $|D|^{-1}$ to both sides implies $|D|^{-1} H_V u = \I z\cH u$, where $\cH$ is the Hilbert transform. This in particular implies that $|D|^{-1} H_V u \in H^1(\R)$ 
and 
\begin{equation*}
	\gt_V[u,v] = (H_Vu, v)_{L^2} = (|D|^{-1} H_Vu, v)_{ \dot{H}^{\frac12}} 
\end{equation*}
for every $v \in H^1(\R)$. Hence $u \in \dom(\cL)$ and $\cL u = |D|^{-1} H_Vu = \I z\cH u = z\cJ u$, which shows that $\ker(H_V-zD)	\subseteq \ker (\cL-z\cJ)$. 
	
Similarly, if $u\in \ker(\cA -z)$, then $u \in \dom(\cL)$ and $\cL u =z\cJ u $. Since $\dom(\cL)\subset H^1(\R)$ and $\cH$ acts as isometry on $H^1(\R)$, we get $\cL u\in H^1(\R)$.  Moreover, 
\[
\gt_V[u,v] = (\cL u, v)_{\dot{H}^{\frac12}} = (|D| \cL u , v)_{L^2}
\]
holds for every $v\in H^1(\R)$. This implies that $u\in \dom(H_V)$ and $H_Vu = |D| \cL u$. It remains to note that  
\[
\gt_V[u,v] = (\cL u, v)_{\dot{H}^{\frac12}} =z(\cJ u,v)_{\dot{H}^{\frac12}}= \I z(|D|\cH  u , v)_{L^2} = z(Du,v)_{L^2(\R)}.
\]
Therefore, $\ker (\cL-z\cJ)	\subseteq \ker(H_V-zD)$, which proves \eqref{eq:gev1}.

Clearly,  \eqref{eq:gev} with $n=1$ is equivalent to \eqref{eq:gev1}. The claim for $n\geq 2$ can easily be proven by using the same steps and an induction argument.
\end{proof}

\begin{corollary}\label{cor:ker=ker}
$\rE_0(A) = \rE_0(\cA)$. In particular,  $\ker(H_V) = \ker(\cL )$.
\end{corollary}

\begin{corollary} \label{lem:indexinv}
	Let $V \in L_\loc^1(\R)$ be a real-valued function satisfying \eqref{eq:cond_int_limit_0}. Then
	\begin{equation}
			(u, H_V v)_{L^2} = (u,  \cL v)_{\dot{H}^{\frac12}}
	\end{equation}
 for every $z \in \C$ and  for all $u$, $v \in \rE_z(A)=\rE_z(\cA)$.
\end{corollary}

\begin{proof}
Let $u, v \in \rE_z(A)=\rE_z(\cA)$. Then $u,v\in \dom(\cL)\cap\dom(H_V)$ and hence 
	\[
\gt_V[u,v]=(u, \cL v)_{\dot H ^{1/2}} = (u, H_V v)_{L^2}.\qedhere
	\]
\end{proof}

Since  $\cL$ and $H_V$ are closely connected, we end up with the following formula.

\begin{corollary}\label{cor:kappaA}
Assume that $V\in L^1_{\loc}(\R)$ is real valued and satisfies \eqref{eq:cond_int_limit_0}. 
\begin{itemize}
\item[{\rm (i)}] If $\ker(H_V)=\{0\}$, then $\kappa_{\rm Ham}(\cA) = \kappa_-(H_V)$. 
\item[{\rm (ii)}] If $\ker(H_V)\neq\{0\}$ and $\bD\neq 0$, then
\be\label{eq:kappaA2}
\kappa_{\rm Ham}(\cA) = \kappa_-(H_V) - \kappa_-(\bD_V),
\ee
where $\bD_V$ is defined by \eqref{eq:bD_V}.
\end{itemize}
\end{corollary}

\begin{proof}
In view of \eqref{eq:kappa-}, it suffices to show that $\bD_V=\bD$. To justify the definition of $\bD_V$, first notice that $\psi_0\in \ker(H_V)$ can be chosen real valued. Next observe that  $\psi_0'\in \ran(H_V)$ if $\psi_0'\perp \ker(H_V)$. However, 
$(\psi_0',\psi_0)_{L^2} = 0$ since $\psi_0\in H^1(\R)$.

Next notice that if $u \in \dom(H_V) $ and $H_V u \in \dom(|D|^{-1})$, then $u \in \dom(\cL)$ and $\cL u = |D|^{-1} H_V u$. Indeed, from \eqref{eq:repH} we get
\begin{equation*}
		\gt_V[u,v] = (H_V u, v)_{L^2} = (|D|^{-1} H_V u, v)_{\dot{H}^{\frac 1 2}}
\end{equation*}
for $v \in H^1(\R)$. Using \eqref{eq:repL} the claim is proven.

Now set $u := H_V^{-1} \psi_0' \in \dom(H_V)$. Then $H_V u = \psi_0' \in \dom(|D|^{-1})$. Thus $u \in \dom(\cL)$ and $\cL u = |D|^{-1} H_V u = |D|^{-1} \psi_0' = \I \cJ \psi_0$. Hence $u- \I \cL^{-1} \cJ \psi_0 \in \ker(\cL) = \ker(H_V)$. Because $\psi_0' \perp \ker(H_V)$ in $L^2(\R)$,
\begin{equation*}
	\bD_V = (H_V^{-1} \psi_0', \psi_0')_{L^2} = \I (\cL^{-1} \cJ \psi_0, \psi_0')_{L^2} = (\cL^{-1} \cJ \psi_0, \cJ \psi_0 )_{\dot{H}^{\frac 1 2}} = \bD.	\qedhere
\end{equation*}
\end{proof}

The spectral instability definition \ref{def:specInst} refers to the operator $H_V$. More precisely, let $\kappa_{\C_+}(A)$ be the total algebraic multiplicity of eigenvalues of \eqref{eq:indSpec} lying in the open upper half-plane $\C_+$.  If $\lambda\in \R\setminus\{0\}$ is an eigenvalue of \eqref{eq:indSpec}, then we shall denote by $\kappa_{\lambda}^-(A)$ its negative index, i.e., the number of non-positive eigenvalues of $P_\lambda H_V P_\lambda$, where $P_\lambda$ denotes the orthogonal projection in $L^2(\R)$ onto $E_\lambda(A)$. The total negative Krein index is then defined by
\be
\kappa_{\R}^-(A) = \sum_{\lambda\in \R\setminus\{0\}} \kappa_\lambda^-(A).
\ee
Notice that the spectral problem \eqref{eq:indSpec} has an additional symmetry. Namely, its point spectrum is symmetric with respect to the imaginary axis since $V$ is real-valued (indeed, if $z\in \C$ is the eigenvalue of \eqref{eq:indSpec} and $\psi_z$ is the corresponding eigenfunction, then $-z^\ast$ is also an eigenvalue and the corresponding eigenfunction is simply $\psi_z^\ast$). Then we can split unstable eigenvalues in three groups and write 
\[
\kappa_{\C_+}(A) = \kappa_{\I\R_{>0}}(A) + 2\kappa_{\rm I}(A),
\]
where $\kappa_{\I\R_{>0}}(A)$ is the total algebraic multiplicity of purely imaginary eigenvalues with positive imaginary part and $\kappa_{I}(A)$ is the total algebraic multiplicity of eigenvalues lying in the first open quadrant. Moreover,
\[
\kappa_{\R}^-(A) = 2 \kappa_{\R_{>0}}^-(A) :=2 \sum_{\lambda>0} \kappa_\lambda^-(A).
\]
The number 
\be
\kappa_{\rm Ham}:= \kappa_{\I\R_{>0}}(A) + 2\kappa_{\rm I}(A) + 2\kappa_{\R_{>0}}^-(A)
\ee
is called {\em the Hamiltonian--Krein index} of \eqref{eq:indSpec}.

Combining Lemma \ref{lem:gev} with Corollary \ref{lem:indexinv}, we get the next result, which completes the proof of Theorem \ref{th:main}.

\begin{corollary}
Let $V \in L_\loc^1(\R)$ be a real-valued function satisfying \eqref{eq:cond_int_limit_0}. Then
\be
\kappa_{\C_+}(A) = \kappa_{\C_+}(\cA),\qquad \kappa_{\R}^-(A) = \kappa_{\R}^-(\cA).
\ee
In particular, the Hamiltonian--Krein index $\kappa_{\rm Ham}(A)$ of \eqref{eq:indSpec} coincides with $\kappa_{\rm Ham}(\cA)$.
\end{corollary}

Hence Corollary \ref{cor:kappaA} implies that
\be\label{eq:indSS}
\kappa_{\I\R_{>0}}(A) + 2\kappa_{\rm I}(A) + 2 \kappa_{\R_{>0}}^-(A) = \kappa_-(H_V)-\kappa_-(\bD).
\ee
Taking \eqref{eq:indSS} into account, we recover and slightly improve the results from \cite{ss}.

\begin{corollary}\label{cor:3.10}
	Let $V \in L_\loc^1(\R)$ be a real-valued function satisfying \eqref{eq:cond_int_limit_0}. If $\dim\ker(H_V)  =1$ and $\kappa_-(H_V)=1$, then 
	\begin{itemize}
	\item[{\rm (i)}] \eqref{eq:indSpec} is spectrally stable if $(H_V^{-1} \psi_0', \psi_0')_{L^2}<0$, where $\ker(H_V) = {\rm span}\{\psi_0\}$.
	\item[{\rm (ii)}] \eqref{eq:indSpec}  is spectrally unstable if  $(H_V^{-1} \psi_0', \psi_0')_{L^2}>0$ and in this case 
	\[
	\kappa_{\rm Ham} = \kappa_{\I\R_{>0}}(A)=1.
	\]
	\item[{\rm (iii)}] If $(H_V^{-1} \psi_0', \psi_0')_{L^2}=0$, then $\kappa_{\C_+}(A)=\kappa_{\R}^-(A)=0$ and $\dim \rE_0(A)\ge 3$.
	\end{itemize}
\end{corollary}

\begin{proof}
	Items (i) and (ii) follow from \eqref{eq:indSS} and it remains to prove (iii). In the proof of Corollary \ref{cor:kappaA} we showed that $\ker(H_V)^\perp = \operatorname{ran}( H_V)$ and $\psi_0' \perp \ker(H_V)$. Hence $\psi_1 := H_V^{-1} \psi_0'$ is well-defined and, moreover,  by \eqref{eq:kerHn}, $\psi_1\in E_0(A)$. Therefore $\dim(E_0(A)) \geq 2$ regardless of the value of $(H_V^{-1} \psi_0', \psi_0')_{L^2}$. 
	
	 Assume now that $(H_V^{-1} \psi_0', \psi_0')_{L^2}=0$. Notice that $\psi_1=H_V^{-1} \psi_0'\in \dom(H_V)$  and hence $\psi_1\in H^1(\R)$. Then we get after integration by parts
	\begin{equation*}
		0 = (H_V^{-1} \psi_0', \psi_0')_{L^2} = (\psi_1, \psi_0')_{L^2} = - (\psi_1', \psi_0)_{L^2}.
	\end{equation*}
	Applying the same procedure, we get that $\psi_2 := H_V^{-1} \psi_1'$ is well-defined and in $E_0(A)$. Thus, $\dim(E_0(A)) \geq 3$ in this case. 
	Finally, by Corollary \ref{cor:3.3}, $k(0)\ge 2$ since $\nu(0)\ge 3$. It remains to notice that $\deg(p)\le 3$ and hence $p$ does not have non-real zeros as well as zeros on $\R\setminus\{0\}$ since $p(z) = zp_0(z)(p_0(z^\ast))^\ast$ with $\deg(p_0)\le 1$. 
\end{proof}

\appendix

\section{Quadratic forms in Hilbert spaces}\label{app:01}

Let $A$ be a self-adjoint lower semibounded operator in $\gH$,
$A=A^*\ge -c $. Denote by $\gt'_A$ the (densely defined)
quadratic form given by 
\[
\gt'_A[f]=(A f,f),\qquad \dom(\gt'_A)=\dom(A).
\]
 It is known (see
\cite{kato}) that this form is closable and lower
semibounded, $\gt'_A\ge -c$. Its closure $\gt_A$ satisfies
$\gt_A\ge -c$. Moreover, by the second
representation theorem \cite[Theorem 6.2.23]{kato}, $\gt_A$
admits the representation
\begin{equation}\label{repr_1}
\gt_A[u]= \|(A+c)^{1/2}u\|^2_\gH -c\|u\|^2_\gH, \qquad
u\in\dom(\gt_A) = \dom\bigl((A+c)^{1/2}\bigr).
\end{equation}
Denote by $\gH_A$ the form domain $\dom(\gt_A)$ equipped  with the
norm
\begin{equation}\label{1.1}
\|u\|_A:=\gt_A[u]+(1+c)\|u\|^2_\gH,\qquad u\in\dom(\gt_A).
     \end{equation}

\begin{definition}\label{def2.2}
The form $\gt$ is called \emph{relatively form bounded} with respect to
$\gt_A$ ($\gt_A$-bounded) if $\dom(\gt_A)\subseteq \dom(\gt)$ and
there are constants $a,b>0$ such that
\begin{equation}\label{II.4}
|\gt[f]|\leq a\gt_A[f]+b\|f\|^2_\gH,\qquad f\in\dom(\gt_A).
\end{equation}
If \eqref{II.4} holds with some $a<1$, then $\gt$ is called {\em strongly $\gt_A$-bounded}. If $a$ can be chosen arbitrary small, then $\gt$ is called \emph{infinitesimally $\gt_A$-bounded}.
\end{definition}

\begin{theorem}[KLMN]\label{th_KLMN}
Let $\gt_A$ be the form corresponding to the operator $A=A^*\ge -c $ in $\gH$. If the form $\gt$ is strongly $\gt_A$-bounded, then the form
\begin{equation}\label{II.5}
{\gt}_1 := {\gt}_A+{\gt}, \qquad   \dom({\gt}_1):=\dom({\gt}_A),
\end{equation}
is closed and lower semibounded in $\gH$ and hence gives rise to a self-adjoint semibounded operator. 
   \end{theorem}

Recall that a quadratic form $\gt$ in $\gH$ is called {\em compact} if it is bounded, $\gt= \gt_C,$
and the (bounded) operator $C$ is compact in $\gH$.

We also need the following result of M.\ Sh.\ Birman (see \cite[Theorem 1.2]{birman}).

  \begin{theorem}[Birman]\label{th2.2}
Let $A=A^*\ge -c $ in $\gH$  and let $\gt_A$
 be the corresponding form. 
If the quadratic form $\gt$ in ${\gH}$  is compact in $\gH_A$ (or
simply, $\gt_A$-compact), then the form $\gt_1$ defined by
\eqref{II.5} is closed, lower semibounded in $\gH$, and the operator $B=B^*$
associated with the form ${\gt}_1$  satisfies
$\sigma_{\ess}(B)=\sigma_{\ess}(A)$.
    \end{theorem}

Notice that the form $\gt$ is  infinitesimally $\gt_A$-bounded if it
is $\gt_A$-compact.      
 We also need the following useful fact.

   \begin{lemma}\label{lem3.1}
Let $A=A^*\ge -cI$ and let $\gt$ be a nonnegative quadratic form in $\gH$ such that $\gH_A \subset
\dom(\gt)$ and $\gt$ is closable  in $\gH_A$. Then the form $\gt$
is compact in $\gH_A$ if and only if the embedding $i\colon
\gH_A\to\dom(\gt)$ is compact.
  \end{lemma}

\end{document}